\documentclass[11pt]{amsart}
\usepackage{geometry}
\usepackage{float}
\usepackage[usenames]{color}
\usepackage{graphicx}
\usepackage{amscd,amssymb,verbatim,xcolor}
\usepackage{longtable, multirow}
\usepackage{soul,cancel}

\date{\today}

\newcommand\bC{{\mathbb C}}

\newcommand\bR{{\mathbb R}}

\newcommand\fa{{\mathfrak a}}

\newcommand\fg{{\mathfrak g}}
\newcommand\fh{{\mathfrak h}}
\newcommand\frk{{\mathfrak k}}
\newcommand\fl{{\mathfrak l}}

\newcommand\fp{{\mathfrak p}}
\newcommand\fq{{\mathfrak q}}
\newcommand\fs{{\mathfrak s}}
\newcommand\ft{{\mathfrak t}}
\newcommand\fu{{\mathfrak u}}
\newcommand\fz{{\mathfrak z}}

\newtheorem{theorem}{Theorem}[section]

\newtheorem{corollary}[theorem]{Corollary}

\newtheorem{definition}[theorem]{Definition}
\newtheorem{example}[theorem]{Example}
\newtheorem{lemma}[theorem]{Lemma}
\newtheorem{proposition}[theorem]{Proposition}
\newtheorem{remark}[theorem]{Remark}

\newcommand\im{{\operatorname{im}}}

\begin{document}
\title[Dirac series of $GL(n, \bR)$]{Dirac series of $GL(n, \bR)$}
\author{Chao-Ping Dong}
\author{Kayue Daniel Wong}

\address[Dong]{Mathematics and Science College, Shanghai Normal University, Shanghai 200234,
P. R. China}
\email{chaopindong@163.com}

\address[Wong]{School of Science and Engineering, The Chinese University of Hong Kong, Shenzhen,
Guangdong 518172, P. R. China}
\email{kayue.wong@gmail.com}

\begin{abstract}
The unitary dual of $GL(n, \bR)$ was classified by Vogan in the 1980s.  In particular, the Speh representations and the special unipotent representations are the building blocks of the unitary dual with half-integral infinitesimal characters. In this manuscript, we classify all irreducible unitary $(\fg, K)$-modules with non-zero Dirac cohomology for $GL(n, \bR)$, as well as a formula for (one of) their spin-lowest $K$-types. Moreover, analogous to the $GL(n,\bC)$ case given in \cite{DW1}, we count the number of the FS-scattered representations of $GL(n, \bR)$.
\end{abstract}



\maketitle
\setcounter{tocdepth}{1}

\section{Introduction}\label{sec:intro}
The notion of Dirac cohomology was first introduced by Vogan in the late 1990s. Since then, it is widely studied under different
perspectives in the representation theory of reductive Lie groups (and beyond). As an example, one of the interesting questions is
its relations with the unitary representations of real reductive Lie groups $G$. More precisely, one would like to classify
all irreducible unitary representations with non-zero Dirac cohomology for $G$. As coined by Huang, we call them the \emph{Dirac series} for $G$.
\smallskip

It is hoped that the Dirac series for $G$ will give a better idea on the structure of the whole unitary spectrum $\widehat{G}$ -- for instance, it is known that all $A_{\fq}(\lambda)$-modules with regular, integral infinitesimal character has non-zero Dirac cohomology. On the other hand, the work of Salamanca-Riba \cite{Sa} states that these modules cover the whole unitary spectrum under the assumption that its infinitesimal character is strongly regular. It would therefore be of interest to explore the Dirac series for $G$ without the regularity condition.
\smallskip

A complete classification of the Dirac series has been obtained for the following Lie groups: complex $F_4$ \cite{DD}, complex $E_6$ \cite{D1}; complex classical Lie groups \cite{BDW}; several real exceptional Lie groups  \cite{DDH,DDY}. Results on $U(p, q)$ have been recorded in \cite{DW2}. For instance, Example 6.3 of \cite{DDY} says that the Dirac series of split real $F_4$ go beyond \emph{elliptic representations}. Thus understanding Dirac series is interesting for representation theory in its own right.
\smallskip

This paper aims to classify the Dirac series for $GL(n, \bR)$.
Although the unitary dual of $GL(n, \bR)$ has been determined by Vogan \cite{V3} in the 1980s,
it is still worthwhile to get a thorough understanding of the $GL(n, \bR)$ case, since it will
be essential in understanding the Dirac series for other real reductive Lie groups.
Moreover, the recent research announcement \cite{BP2} of Barbasch and Pand\v zi\'c suggests
that Dirac cohomology could be used to construct automorphic forms. In the case of $GL(n,\bR)$,
representations in the Dirac series that are not of the form of $A_{\fq}(\lambda)$-modules
(a large class of $A_{\fq}(\lambda)$-modules are \emph{Speh modules}, given in Section \ref{sec-Speh-repn} below) shall have new implications in the theory automorphic forms.

\subsection{Dirac cohomology}
Let $G$ be a Lie group in the Harish-Chandra class \cite{HC} (e.g. $G = GL(n,\bR)$) with Cartan involution $\theta$ and maximal compact subgroup $K$, and $T$ be a maximal torus of $K$. Denote the Lie algebras of $G$, $K$, $T$, etc., by $\fg_0$, $\frk_0$ and $\ft_0$. We shall drop the subscript to stand for the corresponding complexified Lie algebra. Then we have Cartan decomposition $\fg = \mathfrak{k} + \fp$ and the Killing form $B$ of $[\fg,\fg]$. We extend $B$ to the center of $\fg$ so that it is non-degenerate, and still denote this extension by $B$. Then $B$ is positive definite on $\fp$. Let $\fa=Z_{\fp}(\ft)$, the commutator of $\ft$ in $\fp$. Then
$$
\fh=\ft+\fa
$$
is called a fundamental Cartan subalgebra of $\fg$. Let us denote by $\fh^*$ the complex linear dual of $\fh$. Denote by $W(\fg, \fh)$ the Weyl group of the root system $\Delta(\fg, \fh)$. Restricting the roots in $\Delta(\fg, \fh)$ onto $\ft$, we get the root system $\Delta(\fg, \ft)$ which may not be reduced. Let $W(\fg, \ft)$ be the corresponding Weyl group. Note that $W(\fg, \ft)$ is isomorphic to $W(\fg, \fh)^\theta$, the subgroup of $W(\fg, \fh)$ fixed by $\theta$. Fix a positive root system $\Delta^+(\frk, \ft)$, and denote the half sum of roots in it by $\rho_c$.

Let $Z_1, \dots, Z_s$ be an orthonormal basis of $\fp$ under the inner product induced by $B$. Then
$$
D=\sum_{i=1}^{s} Z_i\otimes Z_i \in U(\fg)\otimes C(\fp)
$$
is the \emph{Dirac operator} introduced by Parthasarathy \cite{P1}. Here $U(\fg)$ is the universal enveloping algebra of $\fg$,
and $C(\fp)$ is the Clifford algebra of $\fp$.
\smallskip

Let $X$ be a $(\fg, K)$-module and ${\rm Spin}$ be a spinor module for $C(\fp)$. Then
$X \otimes {\rm Spin}$ is an admissible $(U(\fg)\otimes C(\fp), \widetilde{K})$ module, where $\widetilde{K}$ is
the pin double cover of $K$ given by the pullback of the map:
\begin{equation} \label{eq-pullback}
\begin{CD}
\widetilde{K} @>>> {\rm Pin}(\fp_0) \\
@VVV @VV pr V \\
K @>>Ad> {\rm O}(\fp_0)
\end{CD},
\end{equation}
and $(k, s) \in \widetilde{K} \subset K \times {\rm Pin}(\fp_0)$ acts on $X \otimes {\rm Spin}$ by the diagonal action.
In particular, $D$ acts on $X\otimes {\rm Spin}$, and Vogan defined the \emph{Dirac cohomology} of $X$ \cite{V4} as
\begin{equation}\label{Dirac-cohomology}
H_D(X)=\ker D/(\ker D \cap \im D).
\end{equation}
The Dirac operator $D$ is independent of the choice of the orthonormal basis $\{Z_i\}_{i=1}^s$. Thus it commutes with every element in $\widetilde{K}$ and hence
$H_D(X)$ is an $\widetilde{K}$-module.

In the special case when $X$ is unitary,  the Dirac operator is self-adjoint with respect to a natural Hermitian inner product on $X\otimes {\rm Spin}$. In this case, Dirac cohomology has the following form
$$
H_D(X)=\ker D= \ker D^2.
$$
We refer the reader to Proposition 3.1.6 of \cite{HP2} for a nice formula of $D^2$, which was due to Parthasarathy \cite{P1}.

The Vogan conjecture, proved by Huang and Pand\v zi\'c \cite{HP1}, says that Dirac cohomology (whenever non-zero) is a finer invariant than infinitesimal character.
This result is foundational for the computation of Dirac cohomology. The paper \cite{HP1} deals with connected Lie groups.
Being in the Harish-Chandra class, $G$  may be disconnected and we need Theorem A of \cite{DH}, which slightly extends the main result of \cite{HP1}. 
\begin{theorem}\label{thm-HP}
Let $X$ be an irreducible $(\fg, K)$-module with infinitesimal character $\Lambda\in\fh^*$. If $H_D(X)$ is non-zero, let $\widetilde{\gamma}$ be any highest weight of any $\widetilde{K}$-type occurring in $H_D(X)$. Then there exists $w\in W(\fg, \fh)$ such that
\begin{equation}\label{inf-char-HP}
w \Lambda|_{\fa} = 0, \quad w \Lambda|_{\ft} = \widetilde{\gamma} + \rho_c.
\end{equation}
In particular, we can choose $\Lambda$ such that $\Lambda|_{\mathfrak{t}}$ is dominant for $\Delta^+(\frk, \ft)$, and choose $w$ from $W(\fg, \ft)$. 
\end{theorem}

\subsection{Structure theory of $GL(n,\bR)$}
We will focus on $G = GL(n,\bR)$ henceforth. Let $\theta: GL(n, \bR)\to GL(n, \bR)$ be given by
$$
\theta(g)=(g^{-1})^t,
$$
the inverse transpose of $g\in GL(n, \bR)$. This is the Cartan involution of $GL(n, \bR)$ with
$K=G^{\theta} = O(n)$, whose identity component is $SO(n)$. Embed
$$
T_0=SO(2)\times SO(2) \times \cdots \times SO(2)  \quad (m \mbox{ copies})
$$
into $K$ in the standard way. $T_0$ is a maximal torus of $SO(n)$. The matrix
$$
r_n={\rm diag} \{1, \dots, 1, -1\} \quad (n \mbox{ entries})
$$
normalizes $T_0$.
The Cartan subgroup of $K$ is $T=T_0 \rtimes \{1, r_n\}$.


\subsubsection{Vogan Diagrams}
The semi-simple factor of $\fg$ is $[\fg, \fg]=\fs\fl(n, \bR)$. Let us adopt \emph{the} Vogan diagrams for  $\fs\fl(2m+\epsilon,\bR)$ as in Appendix C of Knapp \cite{Kn}.  They are shown in Table \ref{Fig-sl-2n-Vogan} and Table \ref{Fig-sl-2n+1-Vogan}.  By fixing the Vogan diagrams, we have actually fixed a choice of
$$
\Delta^+(\fg, \ft)=\Delta^+(\frk, \ft) \cup \Delta^+(\fp, \ft).
$$
Let us denote the half sum of roots in $\Delta^+(\fg, \ft)$ by $\rho$. 

When $\epsilon=0$,
$$
\Delta^+(\frk, \ft)=\{e_i\pm e_j \mid 1\leq i<j\leq m\}, \quad
\Delta^+(\fp, \ft)=\Delta^+(\frk, \ft)\cup \{2e_i\mid 1\leq i\leq m\}.
$$
Therefore,
$$
\rho_c=(m-1, m-2, \dots, 0), \quad \rho-\rho_c=(m, m-1, \dots, 1) \in \ft^*.
$$
The Weyl group $W(\fg, \ft)=W(\frk, \ft)\cup s_{2e_m} W(\frk, \ft)$.

\begin{figure}[H]
\centering
\scalebox{0.6}{\includegraphics{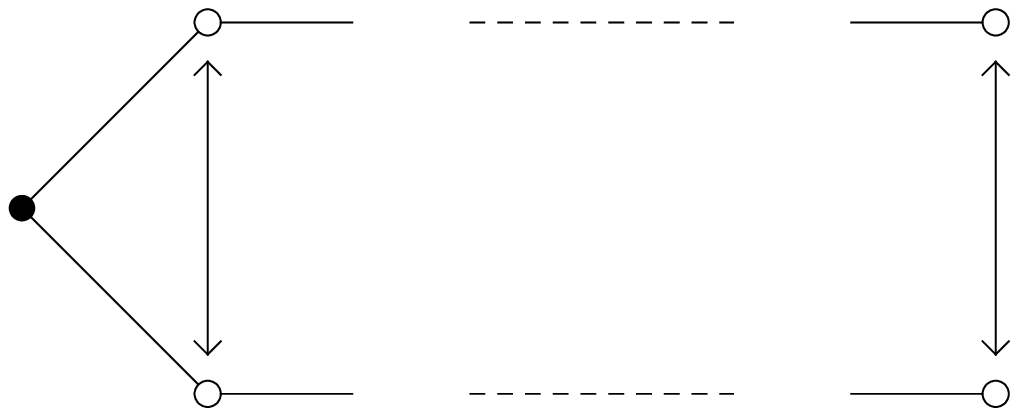}}
\caption{The Vogan diagram for $\fs\fl(2m,\bR)$}
\label{Fig-sl-2n-Vogan}
\end{figure}

When $\epsilon=1$,
$$
\Delta^+(\frk, \ft)=\{e_i\pm e_j \mid 1\leq i<j\leq m\}\cup\{e_1, \dots, e_m\}
$$
and
$$
\Delta^+(\fp, \ft)=\Delta^+(\frk, \ft)\cup \{2e_i\mid 1\leq i\leq m\}.
$$
Therefore,
$$
\rho_c=(m-\frac{1}{2}, m-\frac{3}{2}, \dots, \frac{1}{2}), \quad \rho-\rho_c=(m+\frac{1}{2}, m-\frac{1}{2}, \dots, \frac{3}{2}) \in \ft^*.
$$
The Weyl group $W(\fg, \ft)=W(\frk, \ft)$.

\begin{figure}[H]
\centering
\scalebox{0.6}{\includegraphics{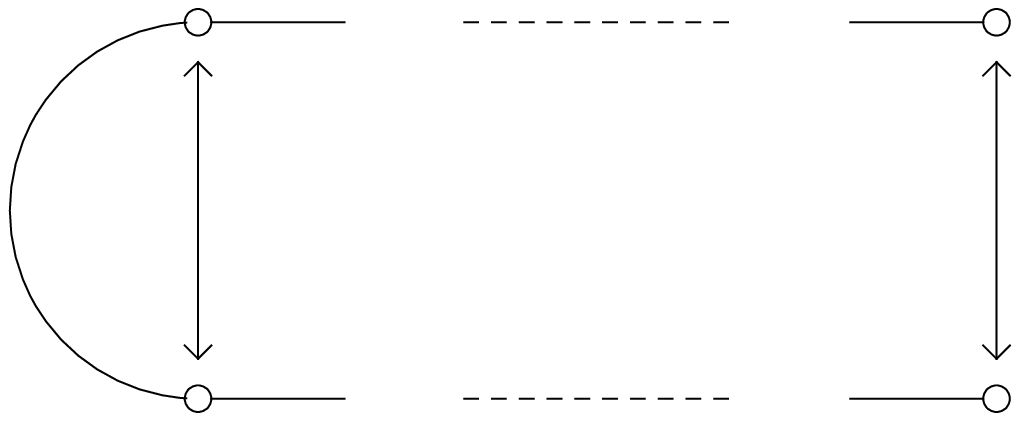}}
\caption{The Vogan diagram for $\fs\fl(2m+1,\bR)$}
\label{Fig-sl-2n+1-Vogan}
\end{figure}

\subsubsection{$K$-types} \label{sec-ktypes}
By the Cartan-Weyl theorem, the equivalence class of irreducible representations of $K = O(n)$ (denoted as $\widehat{K}$)
can be parametrized by irreducible representations of $T$ whose differential is a sum of dominant weights.
Under this point of view, $\widehat{K}$ can be parametrized by pairs $(\gamma; \eta)$, where $\gamma=(\gamma_1, \dots, \gamma_m)$ is a decreasing sequence of non-negative integers (c.f. Proposition 5.7 of \cite{V3}). Moreover, if $n$ is even and $\gamma_m>0$, then $\eta$ is $\frac{1}{2}$; otherwise, $\eta$ can be $0$ or $1$.

Denote by $E_{(\gamma; \eta)}$ the $K$-type parametrized by the pair $(\gamma; \eta)$. If $\eta=0$ or $1$, then the restriction of $E_{(\gamma; \eta)}$ to $SO(n)$ is the irreducible representation with highest weight $\gamma$. If $\eta=\frac{1}{2}$, then the restriction of $E_{(\gamma; \eta)}$ to $SO(n)$ is the sum of two irreducible representations with highest weights $\gamma$ and $(\gamma_1, \dots, \gamma_{m-1}, -\gamma_m)$.

For a vector consisting of integer coordinates, we will refer to the parity of the sum of its coordinates as the \emph{parity} of the vector. In particular,  the parity of the $K$-type $E_{(\gamma; \eta)}$ is defined as that of $\gamma$.

\subsubsection{PRV component}
Since we are interested in the Dirac series for $G = GL(n,\bR)$, it suffices to study
Dirac cohomology for all unitary $(\fg,K)$-modules $X$. Suppose $\widetilde{\gamma}$ is the highest weight
of a $\widetilde{K}$-type appearing in $H_D(X)$. By Parthasarathy's Dirac operator inequality \cite{P2},
it must come from the \emph{PRV component} \cite{PRV} of a $K$-type $E_{(\tau;\eta)}$ in $X$
tensored with $\mathrm{Spin}$ which we describe below.

The spinor module $\mathrm{Spin}$ has a natural $\frk$-module structure given by $\frk \stackrel{\mathrm{ad}}{\longrightarrow} \mathfrak{spin}(\fp) \hookrightarrow C(\fp),$ whose highest weight(s) are of the form:
\begin{equation} \label{eq-sigma}
\sigma = \begin{cases} (m,m-1,\dots,\pm 1) & \text{when}\ n = 2m\ \text{is even}\\
(m+1/2,m-1/2,\dots, 3/2) & \text{when}\ n = 2m+1\ \text{is odd}
\end{cases}.
\end{equation}
Then the PRV component of $E_{(\tau;\eta)} \otimes \mathrm{Spin}$
has highest weight(s) equal to $\{\tau - \sigma\}$, where $\{\chi\}$ is the unique $W(\mathfrak{k},\ft)$-conjugate of
$\chi$ which is $\Delta^+(\mathfrak{k},\ft)$-dominant. In this special case, Theorem \ref{thm-HP} can be rephrased as follows.

\begin{corollary} \label{cor-HP}
Let $G = GL(n,\mathbb{R})$ and $X$ be a unitary $(\fg,K)$-module with infinitesimal character $\Lambda \in \fh^*$. Then
$H_D(X)$ is non-zero if and only if there exists a $w\in W(\fg, \fh)$, and a $K$-type in $X$ with highest weight $\tau$ such that
\begin{equation}  \label{eq-HP2}
w\Lambda|_{\fa} = 0, \quad \{\tau - \sigma\} = w\Lambda|_{\ft} - \rho_c.
\end{equation}
Moreover, the $K$-types satisfying \eqref{eq-HP2} are the {\bf spin lowest $K$-types} of $X$ in the sense of the first named author's thesis. That is, the norm $\|\{\tau-\sigma\}+\rho_c\|$ attains the minimum value among all the $K$-types of $X$.
\end{corollary}

We look more closely to \eqref{eq-HP2}. Suppose $X$ is a unitary $(\fg,K)$-module satisfying Corollary \ref{cor-HP},
whose infinitesimal character $\Lambda =  (\lambda_1, \dots, \lambda_n) \in \mathfrak{h}^*$
is chosen to be $\Delta^+(\fg,\fh)$-dominant. By our choice of the Vogan diagrams above, we have
\begin{equation} \label{eq-restriction}
\begin{aligned}
\Lambda|_{\ft} &= (\lambda_1 - \lambda_n, \lambda_2 - \lambda_{n-1}, \dots, \lambda_m - \lambda_{n-m+1})
\\
\Lambda|_{\fa} &= (\lambda_1 + \lambda_n, \lambda_2 + \lambda_{n-1}, \dots, \lambda_m + \lambda_{n-m+1}, \zeta)
\end{aligned}
\end{equation}
where $\zeta = \lambda_{m+1}$ if $n = 2m+1$ is odd, and the $\zeta$ term is omitted if $n = 2m$ is even.
Then the first part of \eqref{eq-HP2} implies that
$\lambda_i = -\lambda_{n+1-i}$ for all $1 \leq i \leq m$, and $\lambda_{m+1} = 0$ when $n$ is odd.
Therefore, $\Lambda|_{\ft} = (2\lambda_1, \dots, 2\lambda_m) = \{\tau - \sigma\} + \rho_c$
by the second part of \eqref{eq-HP2}.

If $n$ is even, then the coordinates of $\{\tau - \sigma\}$ and $\rho_c$
are all integral. Otherwise, if $n$ is odd, then the coordinates of $\{\tau - \sigma\}$ and $\rho_c$
are all half-integral. In both cases, the coordinates of $\Lambda|_{\ft}$ must be integral for all $n$.

Consequently, in the classification of
the Dirac series for $GL(n, \bR)$, it suffices to focus on infinitesimal
characters $\Lambda \in \mathfrak{h}^*$ of the form $\Lambda|_{\fa} = 0$ and
\begin{equation} \label{eq-HP}
\Lambda|_{\ft} = (2\lambda_1, \dots, 2\lambda_{m})
\end{equation}
where $2\lambda_1 > \dots > 2\lambda_m $ is a sequence of positive (or non-negative) integers when $n$ is odd (or even).
By abuse of notations, we will denote the infinitesimal character $\Lambda \in \fh^*$
by its restriction $\Lambda|_{\ft}$ to $\ft^*$ from now on.

\subsection{Unitary representations of $GL(n,\bR)$}
To achieve the goal of classifying the Dirac series
of $GL(n,\mathbb{R})$, one needs a good understanding of the
unitary dual of $GL(n,\mathbb{R})$, whose building blocks will be recalled
from \cite{V3}.

\subsubsection{Special unipotent representations}
Special unipotent representations play an important role in the study of the unitary dual of reductive Lie groups.
In the $GL(n,\mathbb{R})$ case, we have a nice description of these representations. Namely, they are of the form
\begin{equation} \label{eq-unip}
\pi_u := {\rm Ind}_{\prod_{i=1}^s GL(r_i,\mathbb{R})}^{GL(r,\mathbb{R})}(\bigotimes_{i=1}^s \varphi_i), \quad \varphi_i = \mathrm{triv}\ \text{or}\ \mathrm{sgn}(\det),
\end{equation}
where the Levi factors can be chosen such that $r_1 \geq r_2 \geq \dots \geq r_s\geq 1$, and that $r=\sum_{i=1}^{s}r_i$.
Note that the infinitesimal character of $\pi_u$ is equal to $\Lambda_u = (\chi_1;\dots;\chi_s) \in \ft^*$ (see paragraph after \eqref{eq-HP}), where
$$
\chi_i := (r_i-1, r_i-3, \dots, \epsilon +2 , \epsilon),
$$
where $\epsilon \in \{0,1\}$ satisfies $r_i - 1 \equiv \epsilon (\text{mod}\ 2)$.
Note that these representations are unitary by construction, and \cite{V3} gave a detailed account on its irreducibility.

\subsubsection{Cohomological induction}

Let $\fq=\fl+\fu$ be a $\theta$-stable parabolic subalgebra of $\fg$. Here we arrange the positive roots so that
$$
\Delta^+(\fl, \fh)=\Delta(\fl, \fh)\cap \Delta^+(\fg, \fh), \quad \Delta(\fu, \fh)\subseteq \Delta^+(\fg, \fh).
$$
Denote by $\rho_{L}$ (resp., $\rho({\fu})$)  the half-sum of roots in $\Delta^+(\fl, \fh)$ (resp., $\Delta(\fu, \fh)$).

Let $L$ be the normalizer of $\fq$ in $GL(n, \bR)$. Let $\fz$ be the center of $\fl$. Given an $(\fl, L\cap K)$-module $Z$ with infinitesimal character $\lambda_L$, the cohomological induction functors $\mathcal{R}^{j}(\cdot)$ produce $(\fg, K)$-modules $\mathcal{R}^j(Z)$ for integers $j\geq 0$.
We say that $Z$ is \emph{weakly fair} if
\begin{equation}\label{weakly-fair}
{\rm Re} \langle \lambda_L + \rho(\fu), \alpha|_{\fz} \rangle \ge 0, \quad \forall \alpha\in \Delta(\fu, \fh).
\end{equation}
If all the inequalities in \eqref{weakly-fair} are strict, then $Z$ is said to be in the \emph{fair} range. We say that $Z$ is \emph{weakly good} if
\begin{equation}\label{weakly-good}
{\rm Re} \langle \lambda_L + \rho(\fu), \alpha \rangle \ge 0, \quad \forall \alpha\in \Delta(\fu, \fh).
\end{equation}
If all the inequalities in \eqref{weakly-good} are strict, then $Z$ is said to be in the \emph{good} range.
By \cite{KnV}, suppose $Z$ is a unitary $(\fl, K \cap L)$-module
in the weakly good range, then we have $\mathcal{R}^j(Z) = 0$ for all
$j \neq S = \dim(\fu \cap \fp)$, and $\mathcal{R}^S(Z)$ is unitary irreducible (or zero) with
infinitesimal character equal to $\lambda_L + \rho(\fu)$.

In the special case that $Z=\bC_{\lambda}$ is a one-dimensional unitary character of $L$, we shall call $\mathcal{R}^S(Z)$ an $A_{\fq}(\lambda)$ module. Then the unitary criterion on the previous paragraph can be relaxed to weakly fair range for these modules. We can further arrange that $\lambda_L=\lambda + \rho_{L}$, so that \eqref{weakly-fair} is equivalent to
\begin{equation}\label{weakly-fair-2nd}
{\rm Re} \langle \lambda + \rho(\fu), \alpha \rangle \ge 0, \quad \forall \alpha\in \Delta(\fu, \fh).
\end{equation}
We refer the reader to \cite{KnV} for more materials on cohomological induction.

\subsubsection{Speh representations}\label{sec-Speh-repn}

Another interesting collection of unitary representations of $GL(2k,\mathbb{R})$ are the {\bf Speh representations}. We recall its construction using
the language of cohomological induction.

Let $\lambda = (\frac{b}{2}, \dots, \frac{b}{2}, -\frac{b}{2}, \dots,-\frac{b}{2}) \in \mathfrak{h}^*$ for
integers $b \geq -k$. Then the Speh representations are given by
\begin{equation}\label{Speh-module}
S_{k, b+1} := A_{\mathfrak{q}}(\lambda), \quad \text{where}\ \mathfrak{q} = \mathfrak{l} + \mathfrak{u}\ \text{with}\ \mathfrak{l} =\mathfrak{gl}(k,\mathbb{C}).
\end{equation}
The inequality $b \geq -k$  guarantees that $S_{k, b+1}$ is in the weakly fair range. Indeed, as carefully illustrated in Example 1 on page 586 of \cite{KnV},
$
\lambda + \rho(\fu)=(b+k, b+k, \dots, b+k)\in\ft^*.
$
Then \eqref{weakly-fair-2nd} says that $S_{k, b+1}$ is weakly fair precisely when $b+k\geq 0$.

The infinitesimal character of $S_{k, b+1}$ is equal to
\begin{equation} \label{eq-spehinfl}
\Lambda_s = (b+2k-1, \dots, b+3, b+1) \in \ft^*.
\end{equation}
Moreover, the lowest $K$-type of $S_{k, b+1}$ is equal to $E_{(b+k+1,\dots, b+k+1; \frac{1}{2})}$, and the $K$-type decomposition of
$S_{k, b+1}$ is equal to
\begin{equation}\label{K-Speh}
S_{k, b+1}|_K = \bigoplus_{a_1 \geq \dots \geq a_k \geq 0} E_{(b+k+1+2a_1, \dots, b+k+1+2a_k;\frac{1}{2})}.
\end{equation}
Note that for any $K$-type $E_{(\gamma; \eta)}$ in $S_{k, b+1}$,
the vectors $\{\gamma -\sigma\}$ and $\Lambda_s-\rho_c$ have the same parity.

\subsection{Unitary dual of $GL(n,\mathbb{R})$ with half-integral infinitesimal characters}
As deduced in \eqref{eq-HP}, it suffices to go through all the irreducible unitary representations of $GL(n,\mathbb{R})$ with half-integral infinitesimal characters to sieve out all the Dirac series. This part of the unitary dual does
not contain any Stein complementary series, and it can be cut out as follows.

\begin{theorem}\emph{(Vogan \cite{V3})}
Any irreducible unitary representation of $GL(n,\bR)$ whose infinitesimal character $\Lambda \in \mathfrak{h}^*$
satisfies \eqref{eq-HP2} must be of the form
\begin{equation} \label{eq-unitary2}
\mathrm{Ind}_{\prod_{i=1}^t GL(2k_i,\mathbb{R}) \times GL(r,\mathbb{R})}^G(\bigotimes_{i=1}^t S_{k_i,m_i+1} \otimes \pi_u)
\end{equation}
where $S_{k_i, m_i+1} = A_{\mathfrak{q}_i}(\lambda_i)$ are Speh representations given by \eqref{Speh-module},
and $\pi_u$ is a special unipotent representation of $GL(r, \bR)$ given in \eqref{eq-unip}.
\end{theorem}
For this reason, Speh representations and special unipotent representations
are often called the `building blocks' of the unitary spectrum in the literature.

\section{Dirac Cohomology of Special Unipotent Representations}\label{sec-unipotent}

In this section, we focus on Dirac cohomology of unipotent representations $\pi_u$
given by \eqref{eq-unip}. We divide the discussion into the cases when $n$ is odd and $n$ is even.

\subsection{Case I: $n$ is odd}

In this section, we focus on the case of unipotent representations of $GL(n,\mathbb{R})$
when $n$ is odd. Indeed, a simple application on \eqref{eq-HP} implies the following.

\begin{proposition} \label{prop-unip}
Consider $GL(n,\mathbb{R})$ with $n$ odd. Suppose $\pi_u$ is a unipotent $(\mathfrak{g},K)$-module
with non-zero Dirac cohomology, then it must be equal to $\mathrm{triv}\ \text{or}\ \mathrm{sgn}(\det)$, or:
\begin{equation} \label{eq-unipodd}
\pi_{n,b}^{\epsilon_1,\epsilon_2} := {\rm Ind}_{GL(n-2b+1) \times GL(2b-1)}^{GL(n)}(\mathrm{sgn}(\det)^{\epsilon_1} \otimes\mathrm{sgn}(\det)^{\epsilon_2}), \quad b < \frac{n}{2}
\end{equation}
where $\epsilon_i \in \{0,1\}$ for $i = 1,2$.
\end{proposition}
It is obvious that $\mathrm{triv}$ and $\mathrm{sgn}(\det)$ have nonzero Dirac cohomology. Indeed, it turns out that the
Dirac cohomology of all $\pi_{n,b}^{\epsilon_1,\epsilon_2}$ are also nonzero.

\begin{theorem} \label{thm-oddunip}
All $\pi_{n,b}^{\epsilon_1,\epsilon_2}$ have nonzero Dirac cohomology.
\end{theorem}
\begin{proof}
Let $p = \max\{n-2b+1, 2b-1\}$ and $q = \min\{n-2b-1,2b+1\}$. By Section 7 of \cite{V3}, the $K$-types of $\pi_{n,b}^{\epsilon_1,\epsilon_2}$ are of the form
$$\pi_{n,b}^{\epsilon_1,\epsilon_2}|_K \cong \begin{cases}
\displaystyle \bigoplus_{a_1 \geq \dots \geq a_q \geq 0} E_{(2a_1, \dots, 2a_q, 0,\dots,0; \epsilon')} &\text{if}\ \epsilon_1 + \epsilon_2 \equiv 0 (\mathrm{mod}\ 2) \\
\displaystyle \bigoplus_{a_1 \geq \dots \geq a_q \geq 0} E_{(2a_1+1, \dots, 2a_q+1, 0,\dots,0; \epsilon')} &\text{if}\ \epsilon_1 + \epsilon_2 \equiv 1 (\mathrm{mod}\ 2)
\end{cases},$$
where $\epsilon' = \epsilon_1$ if $q$ is odd, and $\epsilon' = \epsilon_2$ if $q$ is even.

The infinitesimal character of $\pi_{n,b}^{\epsilon_1,\epsilon_2}$ is
\begin{equation} \label{eq-oddinfl}
\Lambda = (\overbrace{p-1, p-3, \dots,q+2}^{(p-q-1)/2} ,\overbrace{q, q-1, \dots, 1}^{q}) \in \ft^*.
\end{equation}
Therefore,
$$\Lambda - \rho_c = (\overbrace{\frac{p-q}{2},\frac{p-q}{2}-1,\dots,\frac{3}{2}}^{(p-q-1)/2},\frac{1}{2},\dots,\frac{1}{2}).$$
On the other hand, the PRV component of any $K$-type in $\pi_{n,b}^{\epsilon_1,\epsilon_2}$ tensored with $\mathrm{Spin}$ has highest weight of the form
\begin{equation} \label{eq-oddprv}
\{\sigma- (2a_1,\dots,2a_q, 0,\dots, 0) \} = \{ (\frac{n}{2} - 2a_1, \dots, \frac{n}{2} - (q-1) - 2a_q, \overbrace{\frac{p-q}{2},\dots, \frac{3}{2}}^{(p-q-1)/2}) \}
\end{equation}
when $\epsilon_1 + \epsilon_2$ is even (or replace $2a_i$ by $2a_i +1$ above if $\epsilon_1 + \epsilon_2$ is odd).

For $1\leq j\leq q$, let $\alpha_j=\frac{n-1}{2} - (j-1)$ and $\beta_j=\alpha_j+1$.
Then $\alpha_j$, $\beta_j$ are of different parity with
$$
\left|\frac{n}{2} - (j-1)  - \alpha_j\right| =\left|\frac{n}{2} - (j-1) - \beta_j\right| = \frac{1}{2}.
$$
By letting $\gamma_j, \delta_j \in \{\alpha_j, \beta_j\}$ such that $\gamma_j$ is odd and $\delta_j$ is even,
one can pick the $K$-types
\begin{equation} \label{eq-spinlktunipodd}
E_{(\gamma_1, \dots, \gamma_q, 0, \dots, 0;\epsilon')}, \quad \text{or}\quad
E_{(\delta_1, \dots, \delta_q, 0,\dots,0;\epsilon')}
\end{equation}
in $\pi_{n,b}^{\epsilon_1,\epsilon_2}|_K$ (depending on the parity of $\epsilon_1 + \epsilon_2$) such that the PRV component of this $K$-type
tensored with $\mathrm{Spin}$ given in \eqref{eq-oddprv} has highest weight
$$\{(\pm \frac{1}{2}, \dots, \pm \frac{1}{2}, \frac{p-q}{2},\dots, \frac{3}{2})\} =  (\frac{p-q}{2}, \dots, \frac{3}{2}, \frac{1}{2}, \dots, \frac{1}{2}),$$
which equals $\Lambda - \rho_c$, i.e., \eqref{eq-HP2} is satisfied. Moreover, this is the unique $K$-type in $\pi_{n,b}^{\epsilon_1,\epsilon_2}$ satisfying this property. Hence
each $\pi_{n,b}^{\epsilon_1,\epsilon_2}$ has nonzero Dirac cohomology with spin-lowest $K$-type given in \eqref{eq-spinlktunipodd}. The result follows.
\end{proof}

\subsection{Case II: $n$ is even}
As in the case when $n$ is odd, we first describe all unipotent representations whose infinitesimal characters satisfy \eqref{eq-HP} when
$n$ is even.

\begin{proposition}
Consider $GL(n,\mathbb{R})$ with $n$ even. Suppose $\pi_u$ is a unipotent representation with nonzero Dirac cohomology,
then it must be equal to $\mathrm{triv}$ or $\mathrm{sgn}(\det)$, or of the form:
\begin{equation} \label{eq-unipeven2}
\pi_{n, min}^{\epsilon_1,\epsilon_2} := {\rm Ind}_{GL(n-1) \times GL(1)}^{GL(n)}(\mathrm{sgn}(\det)^{\epsilon_1} \otimes\mathrm{sgn}(\det)^{\epsilon_2}),
\end{equation}
or
\begin{equation} \label{eq-unipeven3}
\pi_{n, b}^{\epsilon_1,\epsilon_2,\epsilon_3} := {\rm Ind}_{GL(n-2b) \times GL(2b-1) \times GL(1)}^{GL(n)}(\mathrm{sgn}(\det)^{\epsilon_1} \otimes\mathrm{sgn}(\det)^{\epsilon_2} \otimes \mathrm{sgn}(\det)^{\epsilon_3}), \quad b < \frac{n}{2}.
\end{equation}
\end{proposition}

Contrary to the case when $n$ is odd, not all $\pi_{n, min}^{\epsilon_1,\epsilon_2}$
and $\pi_{n, b}^{\epsilon_1,\epsilon_2,\epsilon_3}$ have nonzero Dirac cohomology. We first study the case of
$\pi_{n, min}^{\epsilon_1,\epsilon_2}$.

\begin{theorem} \label{thm-unipeven2}
Consider $GL(n,\mathbb{R})$ with $n=2m$ being even.
Then the unipotent representations $\pi_{n, min}^{\epsilon_1,\epsilon_2}$ in \eqref{eq-unipeven2} have nonzero Dirac cohomology
if and only if $\epsilon_1 + \epsilon_2 \equiv m (\mathrm{mod}\ 2)$.
\end{theorem}
\begin{proof}
As in Theorem \ref{thm-oddunip}, the $K$-types of $\pi_{n, min}^{\epsilon_1,\epsilon_2}$ in \eqref{eq-unipeven2} are of the form:
\begin{equation} \label{eq-nminktypes}
\pi_{n, min}^{\epsilon_1,\epsilon_2}|_K \cong \begin{cases}
\displaystyle \bigoplus_{a \geq 0} E_{(2a,0,\dots,0; \epsilon_1)} &\text{if}\ \epsilon_1 + \epsilon_2 \equiv 0 (\mathrm{mod}\ 2)\\
\displaystyle \bigoplus_{a \geq 0} E_{(2a+1,0,\dots,0; \epsilon_1)} &\text{if}\ \epsilon_1 + \epsilon_2 \equiv 1 (\mathrm{mod}\ 2)
\end{cases}.
\end{equation}
The infinitesimal character of $\pi_{n, min}^{\epsilon_1,\epsilon_2}$ is $\Lambda = (n-2,\dots,4, 2, 0) \in \ft^*$ and hence
$$\Lambda - \rho_c = (m-1, \dots, 2, 1,0).$$

Now for each $K$-type showing up in $\pi_{n, min}^{\epsilon_1,\epsilon_2}$ in \eqref{eq-nminktypes}, the PRV component of $E_{(\alpha,0,\dots,0; \epsilon_1)} \otimes \mathrm{Spin}$ has highest weight equal to
$$\{(\alpha - m, -m +1, \dots, -2,-1)\}.$$
In order for \eqref{eq-HP2} to hold, the above expression be equal to $\Lambda - \rho_c$, and hence $\alpha = m$.
Therefore, if $m$ is odd, then $\alpha$ must be odd, and hence by \eqref{eq-nminktypes},
$\pi_{n, min}^{\epsilon_1,\epsilon_2}$ has nonzero Dirac cohomology if and only if $\epsilon_1 + \epsilon_2$ is odd.
Similarly, if $m$ is even, then $\alpha$ is even and hence $\pi_{n, min}^{\epsilon_1,\epsilon_2}$ has nonzero Dirac cohomology if and only if $\epsilon_1 + \epsilon_2$ is even, so the result follows.
\end{proof}

In the proof above, one notices that the parity of the $K$-types of $\pi_u$
plays an important role in determining its Dirac cohomology. More precisely,
we define the parity of a $(\fg,K)$-module by saying $\pi$ is odd/even if
all $K$-types $E_{(\gamma,\eta)}$ appearing in $\pi|_K$ are odd/even (see last
paragraph of Section \ref{sec-ktypes}).
Then we have the following:
\begin{lemma} \label{lemma-unipeven3}
Let $G = GL(n,\mathbb{R})$ where $n$ is even. Then the parity of its unipotent representations are given by:
\begin{itemize}
\item[(a)] $\pi_{n,min}^{\epsilon_1,\epsilon_2}$ has the same parity as $(\epsilon_1 + \epsilon_2)$.
\item[(b)] $\pi_{n, b}^{\epsilon_1,\epsilon_2,\epsilon_3}$ has the same parity as $(\epsilon_2 + \epsilon_3)$.
\end{itemize}
\end{lemma}
\begin{proof}
(a) follows from \eqref{eq-nminktypes} immediately. For (b), consider
\begin{equation} \label{eq-111ktypes}
\begin{aligned}
\pi_{n,b}^{\epsilon_1,\epsilon_2,\epsilon_3}|_K &\cong {\rm Ind}_{O(n-1) \times O(1)}^{O(n)}({\rm Ind}_{O(n-2b) \times O(2b-1)}^{O(n-1)}({\det}^{\epsilon_1} \otimes {\det}^{\epsilon_2}) \otimes {\det}^{\epsilon_3}) \\
&\cong {\rm Ind}_{O(n-1) \times O(1)}^{O(n)}\left(\bigoplus_{\alpha_1 \geq \dots \geq \alpha_q \geq 0, \, \alpha_j \equiv \epsilon_1+\epsilon_2 (\text{mod}\ 2)} E_{(\alpha_1, \dots, \alpha_q, 0,\dots,0; \epsilon')} \otimes {\det}^{\epsilon_3}\right),
\end{aligned}
\end{equation}
where $q = \min\{n-2b,2b-1\}$, and $\epsilon' = \begin{cases} \epsilon_1 & \text{if}\ q\ \text{is odd};\\
\epsilon_2 & \text{if}\ q\ \text{is even} \end{cases}$.
\medskip

Suppose $E_{(\mu; \eta)}$ is a $K$-type appearing in $\pi_{n,b}^{\epsilon_1,\epsilon_2,\epsilon_3}$. By looking at the $O(n-1)$-factor of the the restricted module $E_{(\mu; \eta)}|_{O(n-1) \times O(1)}$, one must have $\eta = \epsilon'$ and
$$\mu_1 \geq \alpha_1 \geq \mu_2 \geq \dots \geq \mu_q \geq \alpha_q \geq \mu_{q+1} \geq 0 = \mu_{q+2} = \dots = \mu_{m}$$
for some $\alpha_1 \geq \dots \geq \alpha_q \geq 0$. In other words, the last $m - q-1$ coordinates of $\mu$
must be equal to zero.

As for the $O(1)$--component in the restriction, one must have
$$
\sum_{i=1}^{q+1} \mu_i - \sum_{j= 1}^q \alpha_j + \eta \equiv \epsilon_3 \quad (\mathrm{mod}\ 2).
$$
Therefore,
$$
\sum_{i=1}^{q+1} \mu_i \equiv \sum_{j= 1}^q \alpha_j + \epsilon' + \epsilon_3 \quad (\mathrm{mod}\ 2).
$$
We \emph{claim} that $\sum_{j= 1}^q \alpha_j + \epsilon' \equiv \epsilon_2 (\mathrm{mod}\ 2)$, from which (b) follows immediately.
To see so, suppose firstly that $\epsilon_1 + \epsilon_2$ is even. Then
$\sum_{j= 1}^q \alpha_j + \epsilon' \equiv \epsilon' (\mathrm{mod}\ 2)$,
and $\epsilon_1 = \epsilon_2$. Thus $\epsilon' = \epsilon_2$ regardless of the parity of $q$.

On the other hand, if $\epsilon_1 + \epsilon_2$ is odd, then
$$
\sum_{j= 1}^q \alpha_j + \epsilon' \equiv q + \epsilon' \quad (\mathrm{mod}\ 2).
$$
If $q$ is odd, then $\epsilon' = \epsilon_1$ and $q + \epsilon' \equiv 1+ \epsilon_1 \equiv \epsilon_2 \,(\mathrm{mod}\ 2)$.
If $q$ is even, then $\epsilon' = \epsilon_2$ and $q + \epsilon' \equiv \epsilon_2 \,(\mathrm{mod}\ 2)$ as well. Thus the claim holds, and the proof finishes.
\end{proof}

\begin{theorem} \label{thm-unipeven3}
The Dirac cohomology of the unipotent representation
$\pi_{n, b}^{\epsilon_1,\epsilon_2,\epsilon_3}$ in \eqref{eq-unipeven3}
is non-zero if and only if it is of the same parity as $b$, i.e. $\epsilon_2+\epsilon_3\equiv b \, (\mathrm{mod}\ 2)$.
\end{theorem}

\begin{proof}
The representations in $\pi_{n, b}^{\epsilon_1,\epsilon_2,\epsilon_3}$ has infinitesimal character equal to
\begin{equation} \label{eq-eveninfl}
\Lambda = (\overbrace{p-1,p-3,\dots,q+2}^{(p-q-1)/2},\overbrace{q,q-1,\dots,1,0}^{q+1}) \in \ft^*,
\end{equation}
where $p = \max\{n-2b,2b-1\}$, $q = \min\{n-2b,2b-1\}$. Hence
\begin{equation}\label{Lambdamrhoc-thm-unipeven3}
\Lambda - \rho_c = (\overbrace{\frac{p-q-1}{2}, \dots,2,1}^{(p-q-1)/2},\overbrace{0,\dots,0}^{q+1}).
\end{equation}

%
Suppose $E_{(\mu; \eta)}$ is a $K$-type appearing in $\pi_{n,b}^{\epsilon_1 , \epsilon_2 , \epsilon_3}$, then by the arguments
in the previous Lemma, the last $m - q-1$ coordinates of $\mu$
must be zero. Then the PRV component of $E_{(\mu;\eta)} \otimes \mathrm{Spin}$ has highest weight of the form:
\begin{align*}
&\{ (\mu_1, \dots, \mu_{q+1}, 0,\dots,0) - (m, m - 1, \dots, 2,1)\}= \\
&\{(\mu_1 - m, \dots, \mu_{q+1} - m - q, -\frac{p-q-1}{2}, \dots, -2,-1)\}.
\end{align*}
In order for $E_{(\mu;\eta)}$ to be a spin lowest $K$-type of $\pi_u$, the above expression must be equal to $\Lambda - \rho_c$ in \eqref{Lambdamrhoc-thm-unipeven3}.
The only possibility for this to hold is when $\mu_i = m - (i-1)$ for all $1 \leq i \leq q+1$, i.e.,
$$E_{(\mu,\eta)} = E_{(m, \dots, m - q, 0,\dots,0;\eta)} \quad \quad (\text{or}\ E_{(m,\dots,2,1;\frac{1}{2})}\ \text{if}\ q = m -1)$$
We are left to show that $E_{(\mu,\eta)}$ shows up in $\pi_{n,b}^{\epsilon_1 , \epsilon_2 , \epsilon_3}$ only when $b \equiv \epsilon_2 + \epsilon_3 (\text{mod}\ 2)$.
Indeed, there is exactly one choice of non-negative integers $\alpha_1 \geq \dots \geq \alpha_q$ of the same parity such that
$$m \geq \alpha_1 \geq m -1 \geq \alpha_2 \geq \dots \geq m - (q-1) \geq \alpha_q \geq m - q.$$
Therefore, $E_{(\mu,\eta)}$ shows up in $\pi_{n,b}^{\epsilon_1 , \epsilon_2 , \epsilon_3}$ with multiplicity one if
and only if $\mu$ has the same parity as $\pi_{n,b}^{\epsilon_1 , \epsilon_2 , \epsilon_3}$. Note that
the parity of $\mu  = (m,\dots,m-q,0,\dots,0)$ is given by
\begin{align*}
\sum_{i=1}^{q+1} (m - (i-1)) \equiv \frac{(p+1)(q+1)}{2}
\equiv b(n-2b+1)  \equiv b \quad (\text{mod}\ 2),
\end{align*}
and consequently the result follows.
\end{proof}

\section{Dirac Cohomology of Generalized Speh Modules} \label{sec-spehs}
In this section, we study the Dirac cohomology of \eqref{eq-unitary2} when there is no unipotent
representations in the inducing module, i.e., we study unitary modules of the form
$$
\pi_s = \mathrm{Ind}_{\prod_{i=1}^t GL(2k_i,\mathbb{R})}^G(\bigotimes_{i=1}^t S_{k_i,m_i+1}).
$$
Since they are all parabolically induced from Speh representations,
we call them {\bf generalized Speh modules}.

As in \cite{DW1} and \cite{DW2}, we classify these modules by using {\bf s-chains}:
$$
\mathcal{A} := [A, A-2, \dots, a+2, a]_s,
$$
which is a descending sequence of non-negative integers where the difference between consecutive entries are all equal to two.
\begin{definition} \label{def-spehs}
Let $\pi_s = \mathrm{Ind}_{\prod_{i=1}^t GL(2k_i,\mathbb{R})}^{GL(2k,\mathbb{R})}(\bigotimes_{i=1}^t S_{k_i,m_i+1})$
be a generalized Speh module such that, after reordering the Levi factors if necessary, $k_1 + m_1 \geq \dots \geq k_t + m_t$.
The {\bf s-chains associated to $\pi_s$} are given by
$\bigcup_{i=1}^t \mathcal{A}_i$, where
$$\mathcal{A}_i := [m_i + 2k_i - 1, \dots, m_i + 3, m_i  + 1]_s.$$
\end{definition}

By \eqref{eq-spehinfl} and the line below it, the coordinates of the chains of
$\pi_s$ are precisely its infinitesimal character, hence we denote $\Lambda_s = \bigcup_{i=1}^t \mathcal{A}_i$
whenever there is no ambiguity. Moreover, the highest
weight of the lowest $K$-type of $\pi_s$ is given by
\begin{equation} \label{eq-lktspehs}
\bigcup_{i=1}^t \mathcal{L}_i,
\end{equation}
where $\mathcal{L}_i := (k_i + m_i + 1, \dots, k_i + m_i + 1)$. Note that the entries of $\mathcal{L}_i$ are one greater than the average value of $\mathcal{A}_i$. Using the ordering of Levi components in Definition \ref{def-spehs}, the above weight is $\Delta^+(\mathfrak{k},\ft)$-dominant.
\medskip

We give another realization of generalized Speh modules using cohomological induction. By Theorem 17.6 of \cite{V3} with the inclusion of Levi subgroups
$
\prod_{i=1}^t GL(k_i,\mathbb{C}) \subset GL(k,\mathbb{C}) \subset GL(2k,\mathbb{R}),
$
then $\pi_s$ can also be expressed as
\begin{equation} \label{eq-spehs}
\pi_s \cong \mathcal{R}_{\fq}^S(Z),
\end{equation}
where $\fq = \fl + \fu$ is the $\theta$-stable parabolic subalgebra with $\fl = \mathfrak{gl}(k,\mathbb{C})$, and
$$Z = \mathrm{Ind}_{\prod_{i=1}^t GL(k_i,\mathbb{C})}^{GL(k,\mathbb{C})}(\bigotimes_{i=1}^t{\det}^{\nu_i})$$
is a unitary $GL(k,\mathbb{C})$-module corresponding to the chains
$\bigcup_{i=1}^t \mathcal{A}_i^{-k}$ using the notations in \cite{DW1} (here $\mathcal{C}^{h}$ is defined to be the
chain by adding each coordinate of the chain $\mathcal{C}$ by the constant $h$).
In particular, the coordinates $\bigcup_{i=1}^t \mathcal{A}_i^{-k}$ gives the infinitesimal character $\Lambda_Z = \Lambda_s - \rho(\fu)$ of $Z$,
and its lowest $L \cap K$-type is of the form $\bigcup_{i=1}^t \mathcal{L}_i^{-(k+1)}$, i.e. the
entries of $\mathcal{L}_i^{-(k+1)}$ is equal to the average value of the entries of $\mathcal{A}_i^{-k}$ for
each $i$.

\begin{remark} \label{rmk-spehs}
Suppose $\pi_s$ has nonzero Dirac cohomology, then by \eqref{eq-HP}
the coordinates of its infinitesimal character $\Lambda_s = \bigcup_{i=1}^t \mathcal{A}_i$
must be distinct and non-negative. By \eqref{weakly-good} and the paragraph after that, this implies that $\pi_s \cong \mathcal{R}^S_{\fq}(Z)$ in \eqref{eq-spehs} is cohomologically induced in
the weakly good range. In other words, these modules are \underline{not} fully supported (see Theorem \ref{thm-Vogan} below). Therefore, they are not {\bf FS-scattered} in the sense
of \cite{DDY}.
\end{remark}

As we saw in the above remark, generalized Speh modules have a close relation with
unitary representations in $GL(k,\mathbb{C})$. In fact, the resemblance of these two cases goes deeper, as we see in
the theorem below.

\begin{theorem} \label{thm-spehs}
Let $\pi_s \cong \mathcal{R}^S_{\fq}(Z)$ be a generalized Speh module in the form of \eqref{eq-spehs},
such that its corresponding chains $\Lambda_s = \bigcup_{i=1}^t \mathcal{A}_i$ satisfy \eqref{eq-HP}. Then $\pi_s$ has nonzero Dirac cohomology.
More precisely, suppose the lowest $K$-type of $\pi_s$ has highest weight $\bigcup_{i=1}^t\mathcal{L}_i$
given by \eqref{eq-lktspehs}. Then $\pi_s$ has a unique, multiplicity-free spin lowest $K$-type with highest weight
$$\bigcup_{i=1}^t \mathcal{T}_i$$
where each $\mathcal{T}_i$ is obtained from $\mathcal{L}_i$ by applying Algorithm 2.2 of \cite{DW1}.
\end{theorem}

\begin{proof}
Since $Z$ is in the weakly good range, the Dirac cohomology of $\pi_s \cong \mathcal{R}^S(Z)$ can be obtained by Theorem B of \cite{DH} saying that
$$
H_D(\mathcal{R}^S(Z)) \cong \mathcal{R}^S_{\widetilde{K}}(H_D(Z) \otimes \mathbb{C}_{-\rho(\fu \cap \fp)}).
$$

By \cite{BP1} or \cite{DW1}, the Dirac cohomology of the $(\fl, L \cap K)$-module $Z$
 consists of a single $\widetilde{L \cap K}$-type with highest weight
$$(\Lambda_s - \rho(\fu)) - \rho(\fl\cap\frk)$$
as given by Theorem \ref{thm-HP}. Since $\Lambda_s \in \ft^*$ is assumed to satisfy \eqref{eq-HP},
$$
[(\Lambda_s - \rho(\fu)) - \rho(\fl\cap\frk)] -\rho(\fu \cap \fp) + 2\rho(\fu \cap \fp) = \Lambda_s - \rho(\fu\cap\frk) - \rho(\fl\cap\frk) = \Lambda_s - \rho_c
$$
is $\Delta^+(\mathfrak{k},\ft)$-dominant. Therefore,
$\mathcal{R}^S_{\widetilde{K}}(H_D(Z) \otimes \mathbb{C}_{-\rho(\fu \cap \fp)})$ is nonzero and consists of a single $\widetilde{K}$-type
with highest weight $\Lambda_s - \rho_c$. Consequently, $H_D(\mathcal{R}^S(Z))$ is nonzero.

Moreover, from the proof of \cite{DH}, one can also keep track of the spin lowest $K$-types of $\pi_s \cong \mathcal{R}^S(Z)$
from that of $Z$.
Namely,
the lowest $(L \cap K)$-type of $Z$ has highest weight $\bigcup_{i=1}^t \mathcal{L}_i^{-(k+1)}$ by Remark \ref{rmk-spehs}(b).
By the main result of \cite{DW1}, the unique, multiplicity free spin lowest $(L \cap K)$-type of $Z$ has highest weight
$\bigcup_{i=1}^t \mathcal{T}_i^{-(k+1)}.$
Then the spin lowest $K$-type of $\mathcal{R}^S(Z)$ has highest weight
$$
\bigcup_{i=1}^t \mathcal{T}_i^{-(k+1)} + 2\rho(\fu \cap \fp) = \bigcup_{i=1}^t \mathcal{T}_i^{-(k+1)} + (k+1,\dots,k+1) = \bigcup_{i=1}^t \mathcal{T}_i.
$$
The result follows.
\end{proof}

\begin{example}
Consider the generalized Speh module $\pi_s$ in $GL(16,\mathbb{R})$ corresponding to the chains
\begin{align*}
[9 && && 7&& && 5&& && 3 && && 1]_s\\
&& && && && && [4 && && 2 && && 0]_s.
\end{align*}
Then $\pi_s$ has infinitesimal character $\Lambda_s = (9,7,5,4,3,2,1,0)$,
and its lowest $K$-type is equal to $E_{(6,6,6,6,6,3,3,3;\frac{1}{2})}$, i.e., $\mathcal{L}_1 = (6,6,6,6,6)$ and $\mathcal{L}_2 = (3,3,3)$.
Theorem \ref{thm-spehs} implies that
$$\mathcal{T}_1 = (6,6,6,7,8) \quad \mathcal{T}_2 = (1,2,3)$$
and hence the spin lowest $K$-type is equal to $E_{(8,7,6,6,6,3,2,1;\frac{1}{2})}$.
Note that the PRV component of $E_{(8,7,6,6,6,3,2,1;\frac{1}{2})} \otimes \mathrm{Spin}$ has highest weight
$$\{(8,7,6,6,6,3,2,1) - (8,7,6,5,4,3,2,1)\} = \{(0,0,0,0,1,2,0,0)\} = (2,1,0,0,0,0,0,0)$$
which is equal to $\Lambda - \rho_c = (9,7,5,4,3,2,1,0) - (7,6,5,4,3,2,1,0)$. Therefore, \eqref{eq-HP2} is satisfied.\hfill\qed
\end{example}

\section{Dirac Cohomology of General Unitary Modules}
In this section, we study Dirac cohomology of all unitary representations of
$GL(n,\mathbb{R})$ given by
\eqref{eq-unitary2}. As in the previous section, we associate to each unitary representation
in \eqref{eq-unitary2} a union of chains:
$$\Lambda := \bigcup_{i=1}^t \mathcal{A}_i \cup \mathcal{U}$$
where each $\mathcal{A}_i$ corresponds to the Speh module $S_{k_i,m_i+1}$,
and the {\bf u-chain} $\mathcal{U} = [u_1, \dots, u_{\lfloor \frac{r}{2} \rfloor}]_u$
corresponds to the unipotent representation $\pi_u$, whose coordinates are
given by $\Lambda_u$, the infinitesimal character of $\pi_u$ restricted to $\ft$.
Under this setting, its infinitesimal character $\Lambda$ satisfies \eqref{eq-HP}
if and only if all the coordinates of its corresponding chains are distinct
and non-negative (positive if $n$ is odd). We are now ready to state the main theorem of this manuscript.

\begin{theorem}\label{thm-main}
Let $\pi$ be a unitary $(\fg,K)$-module of the form \eqref{eq-unitary2} of $GL(n,\mathbb{R})$. Let $\Lambda = \bigcup_{i=1}^t \mathcal{A}_i \cup \mathcal{U}$ be the chains associated to $\pi$ such that \eqref{eq-HP} is satisfied.
Then $\pi$ has nonzero Dirac cohomology if and only if $\pi_u$ has
nonzero Dirac cohomology.

Moreover, suppose $\bigcup_{i=1}^t \mathcal{T}_i$
is the highest weight of the spin lowest $K$-type of the generalized Speh module corresponding
to the chains $\bigcup_{i=1}^t \mathcal{A}_i$ given in Section \ref{sec-spehs},
and $\mathcal{T}_u$ is the highest weight of the spin lowest $K$-type
of $\pi_u$. Then the spin lowest $K$-type of $\pi$ has highest weight
\begin{equation}\label{kappa}
\mathcal{S}=(\mathcal{S}_1;\dots;\mathcal{S}_t;\mathcal{S}_u)
\end{equation}
where each $\mathcal{S}_i$ and $\mathcal{S}_u$ are obtained from $\mathcal{T}_i$ and $\mathcal{T}_u$
in the same way as  Theorem 3.1 of \cite{DW3}.
\end{theorem}

\begin{example}
Let $\pi$ be a unitary $(\fg, K)$-module of $GL(22,\mathbb{R})$ with associated chains
\begin{align*}
[11 && && 9]_s && && && && [5 && && 3 && && 2 && && 1 && &&0]_u\\
&&[10 && && 8 && && 6 && && 4]_s.
\end{align*}
Then
$$\pi = \mathrm{Ind}_{GL(4) \times GL(8) \times GL(10)}^{G} (S_{2,9} \otimes S_{4,4} \otimes \pi_{10,2}^{\epsilon_1,\epsilon_2,\epsilon_3}).$$
By Theorem \ref{thm-unipeven3} and Theorem \ref{thm-main}, $\pi$ has nonzero Dirac cohomology if and only if
$(\epsilon_1,\epsilon_2,\epsilon_3) = (0,0,0), (1,0,0), (0,1,1), (1,1,1)$.

We consider the case $(\epsilon_1,\epsilon_2,\epsilon_3) = (0,0,0)$. Then $\pi$ has lowest $K$-type
$$(11,11, 8,8,8,8, 0,0,0,0,0; 0).$$
The spin lowest $K$-types of the generalized Speh module and $\pi_{10,2}^{\epsilon_1,\epsilon_2,\epsilon_3}$
have highest weights:
$$
(12,11,8,8,8,7), \quad (5,4,3,2,0).
$$
By Theorem 3.1 of \cite{DW3}, the spin lowest $K$-type of $\pi$ has highest weight
$$\tau = (12,11,9,8,8,7,5,4,3,2,1).$$

Note that $\{\tau - \sigma\}$ is equal to
$$
\{(12,11,9,8,8,7,5,4,3,2,1) - (11,10,9,8,7,6,5,4,3,2,1)\}
= (1,1,1,1,0,0,0,0,0,0,0).
$$
On the other hand, $\Lambda - \rho_c$ is equal to
$$
(11,10,9,8,6,5,4,3,2,1,0) - (10,9,8,7,6,5,4,3,2,1,0) = (1,1,1,1,0,0,0,0,0,0,0).
$$
Therefore $\{\tau - \sigma\} = \Lambda - \rho_c$ and \eqref{eq-HP2} is satisfied.
\hfill\qed
\end{example}

\begin{proof}

By the ``independence of polarization" results in Section 17 of \cite{V3},
one can express
\begin{equation}\label{pi-2nd-way}
\pi \cong \mathcal{R}_{\fq}^S(Z \otimes \pi_u)
\end{equation}
as a cohomologically induced module in the weakly fair range
(and the cohomology vanishes in other degrees),
where
$\fq = \fl + \fu$ is the theta-stable parabolic subalgebra whose Levi component is of the form $\fl = \fl_1 + \fl_2$, with $\fl_1 = \mathfrak{gl}(k,\mathbb{C})$ and $\fl_2 = \mathfrak{gl}(r,\mathbb{R})$; and
$Z$ is the unitary $(\fl_1, L_1 \cap K)$-module corresponding to the chains $\bigcup_{i=1}^t \mathcal{A}_i^{-(k+r)}$ with lowest $K$-type $\bigcup_{i=1}^t \mathcal{L}_i^{-(k+r+1)}$ and spin lowest $K$-type $\bigcup_{i=1}^t \mathcal{T}_i^{-(k+r+1)}$ similar to Section \ref{sec-spehs}.

By Theorem 6.3.12 of \cite{V1}, we have the Blattner-type formula for $\pi \cong \mathcal{R}_{\fq}^S(Z \boxtimes \pi_u)$ given by
\begin{equation}\label{eq-blattner}
\begin{aligned}
\left[\pi: E_{\kappa}\right]_K = \sum_{q,d} (-1)^q \left[\left( Z \boxtimes \pi_u \right) \otimes \mathbb{C}_{2\rho(\fu \cap \fp)} \otimes S^d(\fu \cap \fp): H^q(\fu \cap \mathfrak{k}, E_{\kappa})\right]_{L \cap K},
\end{aligned}
\end{equation}
where we adopt a slight abuse of notation by writing
$$Z \boxtimes \pi_u = (Z \boxtimes \pi_u)|_{L \cap K} =
\bigoplus_{\psi_i \in \widehat{L_i \cap K};\ i = 1,2} m_{\psi_1,\psi_2} (E_{\psi_1}^{L_1 \cap K} \boxtimes E_{\psi_2}^{L_2 \cap K}), \quad m_{\psi_1,\psi_2} \in \mathbb{N}.$$
Since
$\mathbb{C}_{2\rho(\fu \cap \fp)} \cong {\det}^{k+r+1} \boxtimes \mathrm{triv}$ as $L \cap K$-modules,
\eqref{eq-blattner} can be rewritten as
\begin{equation}\label{eq-blattner2}
\begin{aligned}
\left[\pi: E_{\kappa}\right]_K = \sum_{q,d} (-1)^q \left[\left( Z^{\#} \boxtimes \pi_u \right) \otimes S^d(\fu \cap \fp): H^q(\fu \cap \mathfrak{k}, E_{\kappa})\right]_{L \cap K},
\end{aligned}
\end{equation}
where $Z^{\#}=Z\otimes_{\mathbb{C}}{\det}^{k+r+1}$ has associated chains
$
\bigcup_{i} \mathcal{A}_i^{-(k+r) + (k+r+1)} = \bigcup_{i} \mathcal{A}_i^{1}
$, lowest $(L_1 \cap K)$-type
$\bigcup_{i} \mathcal{L}_i^{-(k+r+1) + k+r+1} = \bigcup_i \mathcal{L}_i$
and spin lowest $(L_1\cap K)$-type $\bigcup_{i} \mathcal{T}_i$.

Also, one can see that $\fu \cap \fp$ is equal to
\begin{equation}\label{Smuprimep}
\left(E^{L_1 \cap K}_{(1,0,\dots,0)} \boxtimes E^{L_2 \cap K}_{(1,0,\dots,0)}\right) \oplus \left(E^{L_1 \cap K}_{(2,0,\dots,0)} \boxtimes E^{L_2 \cap K}_{(0,\dots,0)}\right). \end{equation}
Note that the above formulas differ from
Equation (11) in \cite{DW3} for Type B (if $n$ is odd) and Type D (if $n$ is even) by replacing
$E^{L_1 \cap K}_{(1,1,0,\dots,0)}$ with $E^{L_1 \cap K}_{(2,0,\dots,0)}$.

We claim that for all $q > 0$, the summation on the right hand side \eqref{eq-blattner2} is zero
for $E_{\kappa} = E_{(\mathcal{S};\eta)}$, whose proof will be postponed to Proposition \ref{prop-vanish}.
Assuming the claim to be true, \eqref{eq-blattner2} can be simplified to
\begin{align*}
[\pi|_K: E_{(\mathcal{S};\eta)}] &=  \left[\left( Z^{\#} \boxtimes \pi_u \right) \otimes S^d(\fu \cap \fp): H^0(\fu \cap \mathfrak{k}, E_{(\mathcal{S},\eta)}) \right] \\
&=  \left[\left( Z^{\#} \boxtimes \pi_u \right) \otimes S^d(\fu \cap \fp): E^{L_1 \cap K}_{(\mathcal{S}_1,\dots,S_t)} \boxtimes  E^{L_2 \cap K}_{(S_u;\eta)} \right]
\end{align*}
Then the argument similar to Section 3 of \cite{DW3} implies that the last line is equal to $1$, and $\mathcal{S}$ satisfies \eqref{eq-HP2}. Consequently $H_D(\pi) \neq 0$, and $E_{(\mathcal{S};\eta)}$ is a spin lowest $K$-type of $\pi$.

On the other hand, if $H_D(\pi_u') = 0$ (which happens only for $G = GL(n,\mathbb{R})$
with $n$ even), we wish to show that $\pi' \cong \mathcal{R}^S_{\fq}(Z \otimes \pi_u')$ also has
zero Dirac cohomology.

Let $\pi_u$ be of the same `type' as $\pi_u'$, in the sense that
both modules are of the form $\pi_{n,min}^{\epsilon_1,\epsilon_2}$ or $\pi_{n,b}^{\epsilon_1,\epsilon_2,\epsilon_3}$,
yet $\pi_u$ and $\pi_u'$ have different parity. Then $H_D(\pi_u) \neq 0$ by the theorems in Section 2.2,
and the above arguments imply that $E_{(\mathcal{S};\eta)}$ is a spin lowest $K$-type of
$\pi \cong \mathcal{R}^S_{\fq}(Z \otimes \pi_u)$, i.e. $E_{(\mathcal{S};\eta)}$ satisfies
$$\{\mathcal{S} - \sigma\} = \Lambda - \rho_c.$$
Note that $\{ \cdot \}$ does not change the parity of $\mathcal{S} - \sigma$, since all coordinates
inside the bracket are integral.

By expressing $\pi$ and $\pi'$ in the form of \eqref{eq-unitary2} and applying Frobenius reciprocity, one can easily check that $\pi$ and $\pi'$ have opposite parity.
Therefore, for all $K$-types $E_{(\mathcal{S}';\eta')}$ appearing in $\pi'|_K$,
$\{\mathcal{S}' - \sigma\}$ must have different parity as $\{\mathcal{S} - \sigma\} = \Lambda - \rho_c$,
which means \eqref{eq-HP2} cannot be satisfied for any $K$-type appearing in $\pi'$, and we have
$H_D(\pi') = 0$ as desired.
\end{proof}

The proof of Theorem \ref{thm-main} finishes by the following:
\begin{proposition} \label{prop-vanish}
Using the notations in \eqref{eq-blattner2}, we have
\begin{equation} \label{eq-impossible}
\left[\left( Z^{\#} \boxtimes \pi_u \right) \otimes S^k(\fu \cap \fp): H^q(\fu \cap \mathfrak{k}, E_{(\mathcal{S},\eta)}) \right] = 0
\end{equation}
for all $k \geq 0$ and $q > 0$.
\end{proposition}
\begin{proof}
First of all, we recall the construction of $\mathcal{S} = (\mathcal{S}_1,\dots,\mathcal{S}_t;\mathcal{S}_u)$
from $(\mathcal{T}_1,\dots,\mathcal{T}_t)$ and $\mathcal{T}_u$ stated in Theorem 3.1 of \cite{DW3}. Namely, by writing
$$\mathcal{T}_u = (t_1,\dots,t_{\ell};0, \dots, 0),$$
so that  $t_1 \geq \dots \geq t_{\ell} > 0$ are the positive entries of the highest weight
of the spin lowest $K$-type of $\pi_u$ with
$$\ell = \begin{cases}
q & \text{if}\ \pi_u = \pi_{n,b}^{\epsilon_1,\epsilon_2}\ \text{if}\ n\ \text{is odd}\\
1 & \text{if}\ \pi_u = \pi_{n,min}^{\epsilon_1,\epsilon_2}\ \text{if}\ n\ \text{is even}\\
q+1 & \text{if}\ \pi_u = \pi_{n,b}^{\epsilon_1,\epsilon_2,\epsilon_3}\ \text{if}\ n\ \text{is even}
\end{cases},$$
then
$$\mathcal{S}_u = (t_1, \dots, t_{\ell};\nu_{i_j}, \dots, \nu_{i_1}, 0 \dots, 0)$$
and each $\mathcal{S}_{i_k}$ is obtained from $\mathcal{T}_{i_k}$ by adding
the partition $\nu_{i_k}$ at the appropriate rows of $\mathcal{T}_{i_k}$. Otherwise,
if $p \notin \{i_1,\dots,i_j\}$, we have $\mathcal{S}_p = \mathcal{T}_p$.

For all ${\bf v} \in \mathbb{R}^n$, let $| {\bf v} |_1$ be the sum of its coordinates.
Then we have
$$|\mathcal{S}_u|_1 = |\mathcal{T}_u|_1 + R, \quad \quad |(\mathcal{S}_1,\dots,\mathcal{S}_t)|_1 = |(\mathcal{T}_1,\dots,\mathcal{T}_t)|_1 + R$$
where $R := \sum_{k=1}^j |\nu_{i_k}|_1$.
\medskip

By the description of  $H^q(\fu \cap \mathfrak{k}, E_{(\mathcal{S};\eta)})$ for $q > 0$ given in
Lemma 3.9 of \cite{DW3}, suppose
$E^{L_1\cap K}_{(S_1',\dots,S_t')} \boxtimes E^{L_2 \cap K}_{(S_u';\eta)}$ appears in
$H^q(\fu \cap \mathfrak{k}, E_{(\mathcal{S};\eta)})$, then $\mathcal{S}_u'$ must be of
the form
\begin{equation} \label{eq-claim1}
\mathcal{S}_u' = (s_1', \dots, s_{\ell}'; \sigma), \quad \quad \quad |\sigma|_1 \geq R
\end{equation}
(this is due to the fact that
for $\mathcal{S} = (\mathcal{S}_1, \dots , \mathcal{S}_t; \mathcal{S}_u)$,
the coordinates of $(\mathcal{S}_1, \dots , \mathcal{S}_t)$
are larger than or equal to that of $\mathcal{S}_u$). Moreover, we have
\begin{equation} \label{eq-claim2}
|(\mathcal{S}_1',\dots,\mathcal{S}_t')|_1 < |(\mathcal{S}_1,\dots,\mathcal{S}_t)|_1 = |(\mathcal{T}_1,\dots,\mathcal{T}_t)|_1 + R.
\end{equation}

\medskip

We now focus on the $L_2 \cap K$ factor in $( Z^{\#} \boxtimes \pi_u) \otimes S^{\bullet}(\fu \cap \fp)$.
Note that the only term of $S^{\bullet}(\fu \cap \fp)$ given in \eqref{Smuprimep}
that contributes to the $L_2 \cap K$-component is $S^{d}(E^{L_1 \cap K}_{(1,0,\dots,0)} \boxtimes E^{L_2 \cap K}_{(1,0,\dots,0)})$.
By the arguments below Equation (11) in \cite{DW3}, the $L_2 \cap K$-components
in the decomposition of $S^{d}(E^{L_1 \cap K}_{(1,0,\dots,0)} \boxtimes E^{L_2 \cap K}_{(1,0,\dots,0)})$
into irreducible $L \cap K$-types are of the form $E^{L_2 \cap K}_{(\gamma;\eta')}$ with $d \geq |\gamma|_1$.
Consequently, in order
$E^{L_2 \cap K}_{(\mathcal{S}_u';\eta)}$ to show up in the second component of
$( Z^{\#} \boxtimes \pi_u) \otimes S^{\bullet}(\fu \cap \fp)$, one must have
$$[\pi_u \otimes E^{L_2 \cap K}_{(\gamma;\eta')}: E^{L_2 \cap K}_{(S_u';\eta)}] > 0.$$
More precisely, there exists a $L_2 \cap K$-type $E^{L_2 \cap K}_{(u_1, \dots, u_{\ell}; 0, \dots, 0; \eta'')}$
in $\pi_u$ such that
$$[E^{L_2 \cap K}_{(u_1, \dots, u_{\ell}; 0, \dots, 0; \eta'')} \otimes E^{L_2 \cap K}_{(\gamma;\eta')}: E^{L_2 \cap K}_{(S_u';\eta)}] > 0.$$
By \eqref{eq-claim1},  one must have $|\gamma|_1 \geq R$ in order
for above to hold. Hence we have $d \geq |\gamma|_1 \geq R$.

\medskip

Going back to the first component of $L \cap K$, note that the $L_1 \cap K$-components
in the decomposition of $S^{d}(E^{L_1 \cap K}_{(1,0,\dots,0)} \boxtimes E^{L_2 \cap K}_{(1,0,\dots,0)})$
into irreducible $L \cap K$-types are of the form $E^{L_1 \cap K}_{\theta}$ with $|\theta|_1 = d \geq R$.
Therefore, in order for \eqref{eq-impossible} to have non-zero multiplicity, one must have
$$[Z^{\#} \otimes  E^{L_1 \cap K}_{\theta} \otimes S^a(E^{L_1 \cap K}_{(2,0,\dots,0)}): E^{L_1 \cap K}_{(\mathcal{S}_1',\dots,\mathcal{S}_t')}] > 0$$
for some nonnegative integer $a$. All $L_1 \cap K$-types $E^{L_1 \cap K}_{(z_1,\dots,z_k)}$ appearing in
$Z^{\#} \otimes  E^{L_1 \cap K}_{\theta} \otimes S^a(E^{L_1 \cap K}_{(2,0,\dots,0)})$ must satisfy
$$|(z_1,\dots,z_k)|_1 = |(\mathcal{T}_1,\dots,\mathcal{T}_t)|_1 + d + 2a \geq |(\mathcal{T}_1,\dots,\mathcal{T}_t)|_1 + R,$$
yet $|(\mathcal{S}_1',\dots,\mathcal{S}_t')|_1 < |(\mathcal{T}_1,\dots,\mathcal{T}_t)|_1 + R$ by \eqref{eq-claim2}.
Therefore,
$$\left[( Z^{\#} \boxtimes \pi_u) \otimes S^k(\fu \cap \fp) : E^{L_1\cap K}_{(S_1',\dots,S_t')} \boxtimes E^{L_2 \cap K}_{(S_u';\eta)}\right] = 0$$
for all $L \cap K$-types $E^{L_1\cap K}_{(S_1',\dots,S_t')} \boxtimes E^{L_2 \cap K}_{(S_u';\eta)}$
appearing in $H^q(\fu \cap \mathfrak{k}, E_{(\mathcal{S};\eta)})$ for all $q > 0$, and the proposition follows.
\end{proof}

Let us record the following interesting result, which follows from Theorem \ref{thm-oddunip} and Theorem \ref{thm-main}.

\begin{corollary}\label{cor-thm-main}
When $n$ is odd, any irreducible unitary representation of $GL(n, \bR)$ whose infinitesimal character satisfies \eqref{eq-HP2} has non-zero Dirac cohomology.
\end{corollary}

\section{Classification of FS-Scattered Representations}

We recall the notion of interlaced chains as given in \cite{DW3}.

\begin{definition}
\begin{enumerate}
\item Two chains $\mathcal{X}_1 = [A, \dots, a]_{x_1}$, $\mathcal{X}_2 = [B, \dots, b]_{x_2}$ ($x_i = s \mbox{ or } u$)
are {\bf linked} if
the entries of $\mathcal{X}_1$ and $\mathcal{X}_2$ are disjoint, and either one of the following holds:
\begin{itemize}
\item $A > B > a$; or
\item $B > A > b$; or
\item $\{\mathcal{X}_1$, $\mathcal{X}_2\}$ $=$ $\{[2k, \dots, 4, 2]_u, [1]_s\}$.
\end{itemize}
\item We say a union of chains  $\displaystyle \bigcup_{i \in I} \mathcal{X}_i$ is {\bf interlaced} if for each $i \neq j$ in $I$,
there exists indices $i = i_0, i_1, \dots, i_k = j$ in $I$ such that $\mathcal{X}_{i_{l-1}}$ and
$\mathcal{X}_{i_{l}}$ are linked for all $1 \leq l \leq k$ (by convention, we also let the single chain $\mathcal{X}$ to be interlaced). Furthermore, we say the interlaced chains {\bf u-interlaced} if it contains exactly one u-chain.
\end{enumerate}
\end{definition}

\begin{example} \label{eg-interlaced}
For $GL(n,\mathbb{R})$ when $n$ is odd, the u-chains are of the form
$[n-1, \dots, 4,2]_u$ or $[p-1, p-3,\dots, q+2, q, q-1, \dots, 2,1]_u$ by \eqref{eq-oddinfl}.
We now list all u-interlaced chains for $n = 5$:
$$[4,2]_u,\quad \quad \quad    [3,1]_u,\quad \quad \quad    [2,1]_u,\quad \quad \quad   \begin{aligned} & [1]_s \\ [2]_u \end{aligned}$$
As for $n = 7$, the u-interlaced chains are
$$[6,4,2]_u,\quad \quad \quad   [5,3,1]_u,\quad \quad \quad   [4,2,1]_u,\quad \quad \quad   [3,2,1]_u,\quad \quad \quad   \begin{aligned} &&[3]_s \\ [4 && && 2]_u \end{aligned},$$
$$ $$
$$\begin{aligned} && & [1]_s \\ [4 && 2]_u \end{aligned},\quad \quad \quad   \begin{aligned} &&[2]_s \\ [3 && && 1]_u \end{aligned},\quad \quad \quad  \begin{aligned} && [3 && && 1]_s \\ && &&[2]_u \end{aligned}$$
\smallskip

As for $GL(n,\mathbb{R})$ when $n$ is even, the u-chains are of the form
$[n-1, \dots, 3,1]_u$ or $[p-1, p-3,\dots, q+2, q, q-1 \dots, 1,0]_u$ by \eqref{eq-eveninfl}.
We now list all u-interlaced chains for $n = 6$:
$$[5,3,1]_u,\quad \quad \quad   [4,2,0]_u,\quad \quad \quad  [3,1,0]_u,\quad \quad \quad   [2,1,0]_u,$$
$$ $$
$$\begin{aligned} &&[2]_s \\ [3 && && 1]_u \end{aligned},\quad \quad \quad \begin{aligned} &&[1]_s \\ [2 && && 0]_u \end{aligned},\quad \quad \quad   \begin{aligned} && [2 && && 0]_s \\ && &&[1]_u \end{aligned}$$
As for $n = 8$, the u-interlaced chains are
$$[7,5,3,1]_u,\quad \quad \quad   [6,4,2,0]_u,\quad \quad \quad  [5,3,1,0]_u,\quad \quad \quad  [4,2,1,0]_u,\quad \quad \quad   [3,2,1,0]_u,$$
$$ $$
$$ \begin{aligned} &&[4]_s \\ [5 && && 3 && && 1]_u \end{aligned},\quad \quad \quad \begin{aligned} && && &&[2]_s \\ [5 && && 3 && && 1]_u \end{aligned},\quad \quad \quad \begin{aligned} &&[3]_s \\ [4 && && 2 && &&0]_u \end{aligned}$$
$$ $$
$$\begin{aligned} && && &&[1]_s \\ [4 && && 2 && && 0]_u \end{aligned},\quad \quad \quad  \begin{aligned} &&[2]_s \\ [3 && && 1 && &&0]_u \end{aligned}, \quad \quad \quad
\begin{aligned} &&[4 && && 2]_s \\ &&  && [3 && && 1]_u \end{aligned},$$
$$ $$
$$  \begin{aligned} &&[2 && && 0]_s \\ [3 && && 1]_u \end{aligned},\quad \quad \quad \begin{aligned} &&[3&& && 1]_s \\ &&  && [2 && && 0]_u \end{aligned},\quad \quad \quad  \begin{aligned} && [4 && && 2 && && 0]_s \\ && && && &&[1]_u \end{aligned}.$$
\smallskip

Note that in both cases, the number of u-interlaced chains for $n + 2$ is exactly double of that of $n$. Indeed, for
$n \geq 5$, one can obtain all u-interlaced chains for $GL(n + 2,\mathbb{R})$ from those of
$GL(n, \mathbb{R})$ by adding an extra odd or even coordinate to a u-interlaced chain of $GL(n, \mathbb{R})$
using Algorithm 3.6 of \cite{DW1}. For instance, the chain $\begin{aligned} &&[1]_s \\ [2 && && 0]_u \end{aligned}$ in $GL(6,\mathbb{R})$
gives rise to two chains
$$\begin{aligned} &&[\mathbf{3}&& && 1]_s \\ &&  && [2 && && 0]_u\end{aligned}, \quad \quad \begin{aligned} && && &&[1]_s \\ [\mathbf{4} && && 2 && && 0]_u \end{aligned} \quad \quad \text{in}\ \ GL(8,\mathbb{R}).$$
And the chain $\begin{aligned} & [1]_s \\ [2]_u \end{aligned}$ in $GL(5,\mathbb{R})$ gives rise to two chains
$$\begin{aligned} && [\mathbf{3} && && 1]_s \\ && &&[2]_u \end{aligned}, \quad \quad \begin{aligned} && & [1]_s \\ [\mathbf{4} && 2]_u \end{aligned} \quad \quad \text{in}\ \ GL(7,\mathbb{R}).$$
\end{example}

The following result is taken from Paul's lecture \cite{Pa}.

\begin{theorem}\label{thm-Vogan} \emph{(Vogan \cite{V2})}
Let $G$ be a real reductive Lie group as in \cite{ALTV}.  Let $p=(x, \lambda, \nu)$ be the \texttt{atlas} parameter of an irreducible $(\fg, K)$-module $\pi$.
Let $S$ be the support of $x$, and $\fq(x)$ be the $\theta$-stable parabolic subalgebra given by the pair $(S, x)$, with Levi factor $L$.
Then  $\pi$ is cohomologically induced, in the weakly good range,
from an irreducible $(\fl, L\cap K)$-module $\pi_L$ with parameter $p_L=(y, \lambda-\rho(\fu), \nu)$. Here $y$ is the KGB element of $L$ corresponding to the KGB element $x$ of $G$. Moreover, $\fq(x)$ is the \textbf{minimum} $\theta$-stable parabolic subalgebra based on which $p$ can be realized as a cohomologically induced module in the weakly good range.

In particular, $p$ \underline{cannot} be realized as a cohomologically induced module from any proper $\theta$-stable parabolic in the weakly good range if and only if each simple root shows up in the support $S$ of the KGB element $x$. That is, if and only if $p$ is fully supported (FS for short).
\end{theorem}

\begin{theorem} \label{thm-scattered}
Let $\pi$ be an irreducible, unitary representation with nonzero Dirac cohomology. Then
$\pi$ is a FS-scattered representation if and only if its corresponding chains are u-interlaced.
\end{theorem}

\begin{proof}
As mentioned in Remark \ref{rmk-spehs}(b), if the chains corresponding to $\pi$ does
not contain $\mathcal{U}$, then it must be a generalized Speh module which is always
cohomologically induced from a $\theta$-stable parabolic subalgebra in the weakly good
range.

Similarly, if the chains of $\pi$ are not interlaced, then suppose
$$\bigcup_{i=1}^t \mathcal{A}_i \cup \mathcal{U} = \bigcup_{j\in J} \mathcal{C}_j \cup \left(\bigcup_{k\in K} \mathcal{C}_k \cup \mathcal{U}\right) ,$$
where $\bigcup_{k \in K} \mathcal{C}_k \cup \mathcal{U}$ are the longest interlaced chains containing $\mathcal{U}$.
Let $p_j := \# \mathcal{C}_j$, $q_k := \# \mathcal{C}_k$ and $p := \sum_{j \in J} p_j$ and $q := \sum_{k \in K} q_k$. By Theorem 17.6 of \cite{V3}, one can express $\pi$ as
\begin{equation}\label{pi-FS}
\pi \cong \mathcal{R}^S_{\fq'}(Z_1 \boxtimes Z_2),
\end{equation}
where
\begin{itemize}
\item $\fq' = \fl' + \fu'$ is the $\theta$-stable parabolic subalgebra with $\fl' = \mathfrak{gl}(p,\mathbb{C}) + \mathfrak{gl}(2q+r,\mathbb{R})$;
\item $Z_1 = \mathrm{Ind}_{\prod_{j \in J} GL(p_j,\mathbb{C})}^{GL(p,\mathbb{C})}(\bigotimes_{j \in J} {\det}^{\nu_j})$ corresponds to the chains
$\bigcup_{j\in J} \mathcal{C}_j'$, where all $\mathcal{C}_j'$ are translations from $\mathcal{C}_j$ by the same constant.
\item $Z_2$ is the unitary module of $GL(2q+r,\mathbb{R})$ corresponding to the chains
$\bigcup_{k\in K} \mathcal{C}_k' \cup \mathcal{U}$.
\end{itemize}
Since the chains $\bigcup_{j\in J} \mathcal{C}_j$ and $\bigcup_{k\in K} \mathcal{C}_k \cup \mathcal{U}$
are not interlaced, the smallest coordinate of $\bigcup_{j\in J}\mathcal{C}_j$ and the largest coordinate of $\bigcup_{k\in K} \mathcal{C}_k \cup \mathcal{U}$ must differ by at least one. In other words, the difference between the $p^{th}$ coordinate and the $(p+1)^{th}$ coordinate of the infinitesimal character of $\pi$ must be positive. By \eqref{weakly-good}, this implies that the inducing module in \eqref{pi-FS} is in the  good range. Hence it is not FS-scattered by Theorem \ref{thm-Vogan}.

The remaining representations correspond to u-interlaced chains, and one can check that they cannot be expressed as cohomologically induced modules in weakly good range by a similar argument above along with the induction in stages
given by Section 17 of \cite{V3}. Hence they are all FS-scattered representations by Theorem \ref{thm-Vogan}. The result follows.
\end{proof}

\begin{corollary}
Let $S_n$ be the number of FS-scattered representations of $GL(n,\mathbb{R})$. Then $S_n$ is given by:
\begin{itemize}
\item If $n= 2m+1$ is odd, then
$$S_{2m+1} = 3\cdot 2^m\ \text{for}\ m \geq 1.$$
\item If $n = 2m$ is even, then
$$S_2 = 2,  \quad \quad S_4 = 6, \quad  \quad S_{2m} = 2^{m+1}\ \text{for}\ m \geq 3.$$
\end{itemize}
\end{corollary}
\begin{proof}
For $n = 2$, the FS-scattered representations are
$\mathrm{triv}$, $\mathrm{sgn}(\det)$ and hence $S_2 = 2$.

For $3 \leq n \leq 6$,
the FS-scattered representations of $GL(n, \bR)$ and their corresponding u-interlaced chains given in Theorem \ref{thm-scattered}
are listed in Tables \ref{Table-FS-GL3}--\ref{Table-FS-GL6}, where $\eta \in \{0, 1\}$. Their total numbers are counted as follows:
\begin{align*}
&2+4 = 6,\\
&2+2+2 = 6,\\
&2+4+4+2 = 12,\\
&2+2+2+4+2+2+2 = 16.
\end{align*}
Each representation has a unique LKT as well as a unique spin LKT.

For $n >6$, one can apply the algorithm given in Example \ref{eg-interlaced} to
compute the number of u-interlaced chains of $GL(n,\mathbb{R})$ from that of $GL(n-2,\mathbb{R})$.
Consequently, $S_n = 2 S_{n-2}$ for $n > 6$ and the result follows.
\end{proof}
\begin{remark}
By using  \cite{ALTV} and the finiteness result \cite{D2}, we have explicitly computed all the FS-scattered representations of $SL(n, \bR)$ for $n$ up to $10$ via the software \texttt{atlas} \cite{At}. The number of FS-scattered representations of $SL(n, \bR)$ is exactly one half of $GL(n, \bR)$.
\end{remark}

\begin{table}[H]
\caption{FS-Scattered representations of $GL(3, \bR)$}
\begin{tabular}{|c|c|c|c|}
\hline
{\rm Parameters }& {\rm FS-Scattered Representations} & {\rm LKT} & {\rm Spin LKT}\tabularnewline
\hline
$[2]_u$ & $\mathrm{triv}/\mathrm{sgn}(\det)$ & $(0; \eta)$ & $(0; \eta)$\tabularnewline
\hline
$[1]_u$ & $\pi_{3,0}^{\epsilon_1,\epsilon_2}$ & $(0; \eta)$, $(1;\eta)$ & $(2;\eta)$, $(1;\eta)$ \tabularnewline
\hline
\end{tabular}
\label{Table-FS-GL3}
\end{table}

\begin{table}[H]
\caption{FS-Scattered representations of $GL(4, \bR)$}
\begin{tabular}{|c|c|c|c|}
\hline
{\rm Parameters }& {\rm FS-Scattered Representations} & {\rm LKT} & {\rm Spin LKT}\tabularnewline
\hline
$[3,1]_u$ & $\mathrm{triv}/\mathrm{sgn}(\det)$ & $(0,0; \eta)$ & $(0,0;\eta)$ \tabularnewline
\hline
$[2,0]_u$ & $\pi_{4,min}^{(0,0)/(1,1)}$ & $(0,0; \eta)$ & $(2,0; \eta)$\tabularnewline
\hline
$[1,0]_u$ & $\pi_{4,1}^{(0,0,1)} \cong \pi_{4,1}^{(0,1,0)}$, $\pi_{4,1}^{(1,0,1)}\cong \pi_{4,1}^{(1,1,0)}$ & $(1,0; \eta)$ & $(2,1;\frac{1}{2})$ \tabularnewline
\hline
\end{tabular}
\label{Table-FS-GL4}
\end{table}

\begin{table}[H]
\caption{FS-Scattered representations of $GL(5, \bR)$}
\begin{tabular}{|c|c|c|c|}
\hline
{\rm Parameters }& {\rm FS-Scattered Representations} & {\rm LKT} & {\rm Spin LKT}\tabularnewline
\hline
$[4,2]_u$ & $\mathrm{triv}/\mathrm{sgn}(\det)$ & $(0,0; \eta)$ & $(0,0; \eta)$\tabularnewline
\hline
$[3,1]_u$ & $\pi_{5,0}^{\epsilon_1,\epsilon_2}$ & $(0,0; \eta)$, $(1,0;\eta)$ & $(2,0;\eta)$, $(3,0;\eta)$ \tabularnewline
\hline
$[2,1]_u$ & $\pi_{5,1}^{\epsilon_1,\epsilon_2}$ & $(0,0;\eta)$, $(1,1;\eta)$ & $(2,2;\eta)$, $(3,1;\eta)$ \tabularnewline
\hline
$[1]\cup [2]_u$ & ${\rm Ind}_{GL(2)\times GL(3)}^{G}(S_{1,1} \otimes \mathrm{triv}/\mathrm{sgn}(\det))$ & $(2,0;\eta)$ & $(3,1;\eta)$\tabularnewline
\hline
\end{tabular}
\label{Table-FS-GL5}
\end{table}

\begin{table}[H]
\caption{FS-Scattered representations of $GL(6, \bR)$}
\begin{tabular}{|c|c|c|c|}
\hline
{\rm Parameters }& {\rm FS-Scattered Representations} & {\rm LKT} & {\rm Spin LKT}\tabularnewline
\hline
$[5,3,1]_u$ & $\mathrm{triv}/\mathrm{sgn}(\det)$ & $(0,0,0; \eta)$ & $(0,0,0;\eta)$ \tabularnewline
\hline
$[4,2,0]_u$ & $\pi_{6,min}^{(0,1)/(1,0)}$ & $(1,0,0; \eta)$ & $(3,0,0; \eta)$\tabularnewline
\hline
$[3,1,0]_u$ & $\pi_{6,1}^{(0,0,1)} \cong \pi_{6,1}^{(0,1,0)}$, $\pi_{6,1}^{(1,0,1)}\cong \pi_{6,1}^{(1,1,0)}$ & $(1,0,0; \eta)$ & $(3,2,0;\eta)$ \tabularnewline
\hline
$[2,1,0]_u$ & $\pi_{6,2}^{(0,0,0)/(1,0,0)/(0,1,1)/(1,1,1)}$ & $(0,0,0;\eta)$, $(1,1,0;\eta)$ & $(3,2,1;\frac{1}{2})$ \tabularnewline
\hline
$[2] \cup [3,1]_u$ & ${\rm Ind}_{GL(2)\times GL(4)}^{G}(S_{1,2} \otimes \mathrm{triv}/\mathrm{sgn}(\det))$ & $(3,0,0;\eta)$ & $(4,1,0;\eta)$\tabularnewline
\hline
$[1] \cup [2,0]_u$ & ${\rm Ind}_{GL(2)\times GL(4)}^{G}(S_{1,1} \otimes \pi_{4,min}^{(0,0)/(1,1)})$ & $(2,0,0;\eta)$ & $(3,2,1;\frac{1}{2})$\tabularnewline
\hline
$[2,0] \cup [1]_u$ & ${\rm Ind}_{GL(4)\times GL(2)}^{G}(S_{2,0} \otimes \mathrm{triv}/\mathrm{sgn}(\det))$ & $(2,2,0;\eta)$ & $(3,2,1;\frac{1}{2})$\tabularnewline
\hline
\end{tabular}
\label{Table-FS-GL6}
\end{table}

\bigskip
\centerline{\scshape Acknowledgements}
We are heartily grateful to Professor Vogan for guiding us through his 1986 paper, which is still highly non-trivial for the younger minds even up to today. We thank Professor Paul for helping us with the software \texttt{atlas}. The second author would also like to
thank the Institute of Mathematical Sciences at the Chinese University of Hong Kong for providing an excellent working environment during the epidemic.

\medskip
\centerline{\scshape Funding}
Dong was supported by the National Natural Science Foundation of China (grant 11571097, 2016--2019). Wong is supported by the Presidential Fund of the Chinese University of Hong Kong (Shenzhen), and the National Natural Science Foundation of China (grant 11901491).

\end{document}